\documentclass[reqno,oneside]{amsart}
\usepackage{amsmath}
\usepackage{amssymb}
\usepackage{amsthm}
\usepackage[foot]{amsaddr}
\usepackage{mathrsfs}
\usepackage{hyperref}
\hypersetup{
    colorlinks=true,
    linkcolor=black,
    citecolor=black
    }
    
\usepackage{tikz}
\usepackage{tikz-3dplot}

\newtheorem{theorem}{Theorem}
\newtheorem{proposition}[theorem]{Proposition}
\newtheorem{corollary}[theorem]{Corollary}
\newtheorem{conjecture}[theorem]{Conjecture}
\theoremstyle{remark}
\newtheorem{definition}[theorem]{Definition}
\newtheorem{remark}[theorem]{Remark}
\newtheorem{example}[theorem]{Example}

\begin{document}

\title{On the optimal arrangement of $2d$ lines in $\mathbb{C}^d$}
\author{Kean~Fallon
\qquad
Joseph~W.~Iverson}
\address{Department of Mathematics, Iowa State University, Ames, IA}

\email{keanpf@iastate.edu}
\email{jwi@iastate.edu}


\maketitle

\begin{abstract}
We show the optimal coherence of $2d$ lines in $\mathbb{C}^{d}$ is given by the Welch bound whenever a skew Hadamard of order~$d+1$ exists.
Our proof uses a variant of Hadamard doubling that converts any equiangular tight frame of size $\tfrac{d-1}{2} \times d$ into another one of size $d \times 2d$.
Among $d < 150$, this produces equiangular tight frames of new sizes when $d = 11$, $35$, $39$, $43$, $47$, $59$, $67$, $71$, $83$, $95$, $103$, $107$, $111$, $119$, $123$, $127$, $131$, and $143$.
\end{abstract}


\section{Introduction}

A \textit{projective code} is an arrangement of $n$ lines (1-dimensional subspaces) in $\mathbb{C}^d$.
An \textit{optimal} projective code makes the least pairwise distance (defined below) as big as possible.
This is a natural geometric problem that arises in applications such as wireless communication~\cite{LHS:03}, multiple description coding~\cite{SH:03}, and compressed sensing~\cite{T:08}.
In this short article, we explain how to double an optimal code of $n \approx 2d$ lines in $\mathbb{C}^d$ to produce another optimal code of $2n$ lines in $\mathbb{C}^n$.
Applying this technique, we construct an optimal arrangement of $2d$ lines in $\mathbb{C}^d$ whenever there is a skew Hadamard of order $d+1$.

By definition, the distance between lines spanned by unit column vectors $f_1,f_2 \in \mathbb{C}^d$ is $\tfrac{1}{\sqrt{2}}\| f_1 f_1^* - f_2 f_2^* \|_F = \sqrt{1 - |f_1^* f_2|^2}$.
Intuitively, the lines form an angle with cosine $|f_1^* f_2|$, which is big when the distance is small.
For $n \geq 2$ lines spanned by unit vectors $f_1,\dotsc,f_n \in \mathbb{C}^d$, the \textbf{coherence} $\max_{i \neq j} | f_i^* f_j|$ gives the cosine of the sharpest angle.
It is small when the least pairwise distance is big.

Several lower bounds on coherence are known~\cite{JKM:19}, but for $\tfrac{n}{d} \approx 2$ the largest is the \textbf{Welch bound}~\cite{W:74}, also known as the \textit{first Rankin bound}~\cite{CHS:96,R:55}, the \textit{simplex bound}~\cite{CHS:96}, the \textit{first Levenstein bound}~\cite{L:82}, or (in the context of equiangular lines) the \textit{relative bound}~\cite{LS:73}.

\begin{proposition}[Welch bound]
Given $n>d$ lines in $\mathbb{C}^d$, arrange unit-norm vector representatives into columns of a matrix
\[
F = \left[ \begin{array}{ccc} f_1 & \cdots & f_n \end{array} \right] \in \mathbb{C}^{d \times n}.
\]
Then the lines spanned by $f_1,\dotsc,f_n$ have coherence at least
\begin{equation}
\label{eq:mu}
\mu = \sqrt{ \frac{ n - d }{d (n-1) } }.
\end{equation}
Equality is achieved if and only if:
	\begin{itemize}
	\item[(i)]
	$| f_i^* f_j | = \mu$ whenever $1 \leq i \neq j \leq n$, and
	\smallskip
	\item[(ii)]
	$FF^* = \tfrac{n}{d} I$.
	\end{itemize}
\end{proposition}

If the Welch bound is achieved, then $F$ is known as an \textbf{equiangular tight frame} (ETF) of size $d \times n$, by virtue of~(i) and~(ii).
ETFs are also known as \textit{Welch bound equality codebooks}.
When an ETF exists, its columns span lines with optimally small coherence.
A variety of ETF constructions are known, many based on combinatorial designs and/or representation theory, as summarized in~\cite{FM:15}.
Involved in several constructions, a \textbf{Hadamard} of order $m$ is an $m\times m$ matrix $H$ with entries in $\{1,-1\}$ that satisfies $H^\top H = mI$.
A \textbf{skew Hadamard} is a Hadamard $H=C+I$ for which $C^\top = -C$.
Then $C^\top C = (m-1)I$, and so $C$ is a skew-symmetric \textbf{conference matrix}.
Our main result creates new ETFs from skew Hadamards.

\begin{theorem}[Main result]
\label{thm:main}
Suppose there is a skew Hadamard of order $d+1$.
Then there is an ETF of size $d \times 2d$.
Thus, the Welch bound gives the optimal coherence of $2d$ lines in $\mathbb{C}^d$.
\end{theorem}

An explicit construction appears in Corollary~\ref{cor:main const}.
By way of comparison, the following is well known.

\begin{proposition}[\cite{DGS:75}]
\label{prop:m/2}
An ETF of size $d \times 2d$ exists whenever there is a skew Hadamard of order $2d$.
\end{proposition}

A skew Hadamard matrix of order $m$ exists only if $m=1$, 2, or a multiple of~4.
Thus, $d = 3 \bmod 4$ in Theorem~\ref{thm:main} and $d$ is even in Proposition~\ref{prop:m/2} (or else $d = 1$).
Skew Hadamards are known to exist for infinitely many orders, including every $m=4k$ with {$k \leq 97$}.
(See~\cite{D:23,KS:08,CK:07}.)
By Theorem~\ref{thm:main}, a complex ETF of size $d \times 2d$ exists for every $d \leq 387$ with $d = 3 \bmod 4$.
Comparing with~\cite{FM:15} (which restricts to $d < 150$), we have new ETF sizes for the following:
\[
d = 11,\,  35,\,  39,\,  43,\,  47,\,  59,\,  67,\,  71,\,  83,\,  95,\, 103,\,  107,\,  111,\, 
119,\, 123,\,  127,\, 131,\, 143.
\]
For every $d$ above, the optimal coherence of $2d$ lines in $\mathbb{C}^d$ was unknown before~now.
Corresponding ETFs are included as ancillary files with the arXiv version of this paper.

A strong version of the Hadamard conjecture asserts that skew Hadamards exist for every order divisible by~4~\cite{W:71b}.
If this is true, then the Welch bound is attained by some set of $2d$ lines in $\mathbb{C}^d$ whenever $d$ is even (by Proposition~\ref{prop:m/2}) and whenever $d = 3 \bmod 4$ (by Theorem~\ref{thm:main}).
This would leave only $d = 1 \bmod 4$.
Furthermore, examples of $d \times 2d$ ETFs are known for infinitely many $d = 1 \bmod 4$, including all such $d < 150$ except the following~\cite{FM:15}:
\[
d = 17,\, 29,\, 53,\, 65,\, 73,\, 77,\, 81,\, 89,\, 93,\, 101,\, 105,\, 109,\, 125,\, 
133,\, 137,\, 149.
\]
(An example for $d = 33$ was recently discovered~\cite{G:21}.)
This suggests the following.

\begin{conjecture}
\label{conj:dx2d}
The Welch bound gives the optimal coherence of $2d$ lines in $\mathbb{C}^d$ for every $d$.
That is, for every $d$ there exists a complex ETF of size $d \times 2d$.
\end{conjecture}

After formulating this conjecture, the authors searched for and located a numerical ETF of size $17 \times 34$ using alternating projections~\cite{TDHS:05}.

\section{Preliminaries}
Any ETF $F \in \mathbb{C}^{d \times n}$ is determined up to a unitary by its Gram matrix $F^* F$, which has ones on the diagonal and off-diagonal entries of modulus $\mu$ given by~\eqref{eq:mu}.
Then $F$ can be recovered up to a unitary by the phases of the off-diagonal entries in $F^*F$, as recorded by the \textbf{signature matrix} $S = \tfrac{1}{\mu}(F^* F - I)$.

\begin{proposition}[\cite{HP:04}]
\label{prop:sig}
~
\begin{itemize}
\item[(a)]
If $S$ is the signature matrix of an ETF with size $d \times n$,
then $S=S^*$ has all zero diagonal entries and unimodular off-diagonal entries, and it satisfies $S^2 = cS + (n-1) I$ for
\begin{equation}
\label{eq:c}
c = (n-2d) \sqrt{ \frac{ n-1 }{ d (n-d) } }.
\end{equation}
\item[(b)]
If $S=S^*$ has all zero diagonal entries and unimodular off-diagonal entries, and if it satisfies $S^2 = cS+(n-1)I$ for some $c$, then it is the signature matrix of some ETF with size $d \times n$, where $d$ is uniquely determined by~\eqref{eq:c}.
\end{itemize}
\end{proposition}

In particular, if $S$ is the signature matrix of a $d \times n$ ETF~$F$, then $-S$ is the signature matrix of some $(n-d) \times n$ ETF~$G$.
Any such $G$ is called a \textbf{Naimark complement} of $F$, and it satisfies $G^* G = I - \nu S$ for
\begin{equation}
\label{eq:nu}
\nu = \sqrt{ \frac{d}{(n-d)(n-1)} }.
\end{equation}
Since $F^* F = I + \mu S$, the $n \times n$ matrix
\[
U = \left[ \begin{array}{c} \sqrt{ \tfrac{d}{n} } F \\[8 pt] \sqrt{ \tfrac{n-d}{n} } G \end{array} \right]
\]
satisfies $U^*U = I$ and is unitary.
Expanding $UU^* = I$, one finds the relations
\[
FG^* = 0
\qquad
\text{and}
\qquad
GF^* = 0.
\]

\section{Doubling equiangular tight frames}
To prove the main result, we introduce a technique involving signature matrices that converts an ETF of size $d \times n$ with $\tfrac{n}{d} \approx 2$ into another ETF of size $n \times 2n$.
For inspiration, consider a proof of Proposition~\ref{prop:m/2}.
If $H = C+I$ is a skew Hadamard of order~$m$, then $C^2 = -(m-1)I$ and $S = iC$ is the signature matrix of an $\tfrac{m}{2} \times m$ ETF, by Proposition~\ref{prop:sig}.
To obtain an $m \times 2m$ ETF, one can apply the \textit{Hadamard doubling} construction of~\cite{W:71a}, which converts $H$ into a new skew Hadamard
\[
\left[ \begin{array}{cc} C+I & C+I \\ C-I & -C+I \end{array} \right]
\]
of order~$2m$.
In effect, Hadamard doubling converts $S$ into a new signature matrix
\[
\left[ \begin{array}{cc} S & S+iI \\ S-iI & -S \end{array} \right]
\]
of an $m \times 2m$ ETF.
As we now show, this procedure generalizes even for signature matrices not given by skew Hadamards.

\begin{theorem}[ETF doubling]
\label{thm:doubling}
Suppose $n > d$ and $|c| \leq 1$ in~\eqref{eq:c}.
If there is a real or complex ETF of size $d \times n$, then there is a complex ETF of size $n \times 2n$.
Specifically, the following hold for $\epsilon = \pm 1$ and $\beta = -c + \epsilon i\sqrt{1 - c^2}$.
\begin{itemize}
\item[(a)]
If $S$ is the signature matrix of a real or complex ETF with size $d \times n$, then
\begin{equation}
\label{eq:Sigma}
\Sigma=\left[ \begin{array}{cc} S & S + \beta I \\ S + \overline{\beta} I & -S \end{array} \right]
\end{equation}
is the signature matrix of a complex ETF with size $n \times 2n$.
\smallskip
\item[(b)]
If $F$ is a $d \times n$ ETF with signature matrix $S$ and Naimark complement $G$, then there exist constants ${a,b>0}$ and $w,z \in \mathbb{C}$ such that
\[
\Phi =
\left[ \begin{array}{cc}
a F & w F \\
b G & z G
\end{array} \right]
\]
is an ETF of size $n \times 2n$.
Namely, for $\lambda = \tfrac{1}{\sqrt{2n-1}}$ one can take
\[
a = \sqrt{ \frac{ \nu + \lambda }{ \mu + \nu } },
\quad
b = \sqrt{ \frac{ \mu - \lambda }{ \mu + \nu } }, 
\quad
w = \frac{ \lambda(n-d +\mu \beta d) }{a \mu n},
\quad
z = \frac{-\lambda\big( d - \nu \beta (n-d) \big)}{b \nu n}
\]
with $\mu$ as in~\eqref{eq:mu} and $\nu$ as in~\eqref{eq:nu}, and then $\Phi$ has signature matrix $\Sigma$ in~\eqref{eq:Sigma}.
\end{itemize}

\end{theorem}

The hypothesis of Theorem~\ref{thm:doubling} is satisfied whenever $d = \tfrac{n}{2}$ for $n$ even, and whenever $d = \tfrac{n \pm 1}{2}$ for $n$ odd.
In the former case, $c =0$, and in the latter case,
$c = \mp \frac{2}{\sqrt{n+1}}$.
Beyond these two cases, the hypothesis of Theorem~\ref{thm:doubling} is not satisfied by any other size of a known ETF listed in the tables of~\cite{FM:15}.
(We do not know if the Welch bound can be attained when $|c| \leq 1$ and $d \notin \{ \tfrac{n}{2}, \tfrac{n \pm 1}{2} \}$, and we leave this as an open problem.)
The condition $|c| \leq 1$ imposes a narrow bound on $t = \tfrac{n}{d} - 2$.
Writing
\[
c = (\tfrac{n}{d} - 2) \sqrt{ \frac{ n - 1}{ \tfrac{n}{d} - 1 } } = t \sqrt{ \frac{ n - 1 }{ t + 1 } },
\]
the hypothesis $c^2 \leq 1$ rearranges to say $(n-1)t^2 - t - 1 \leq 0$.
Solving the quadratic inequality shows $|c| \leq 1$ if and only if
\begin{equation}
\label{eq:window around 2}
- \frac{\sqrt{4n-3} - 1}{2(n-1)} \leq \frac{n}{d} - 2 \leq \frac{\sqrt{4n-3}+1}{2(n-1)}.
\end{equation}
(Roughly, $|\frac{n}{d} - 2| \leq \frac{1}{\sqrt{n}}$.)
The small window around~2 plays a crucial role in our proof of the main result, which applies Theorem~\ref{thm:doubling} to ETFs with $d = \tfrac{n-1}{2}$, that is, $\frac{n}{d} - 2 = \frac{2}{n-1}$.

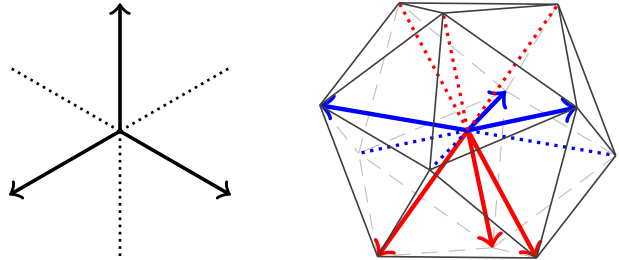
\begin{figure}[b]
\begin{center}
\resizebox{!}{98 pt}{
\begin{tikzpicture}
	\coordinate (0) at (0,0);
	\coordinate (1) at (0,1);
	\coordinate (2) at (-0.866025403784439,-0.5);
	\coordinate (3) at (0.866025403784439,-0.5);
	\coordinate (4) at (0,-1);
	\coordinate (5) at (0.866025403784439,0.5);
	\coordinate (6) at (-0.866025403784439,0.5);
	
	\coordinate (left) at (-1,0);
	\coordinate (right) at (1,0);
	
	\draw [white] (left) -- (right);
	
	\draw [thick, ->] (0) -- (1);
	\draw [thick, ->] (0) -- (2);
	\draw [thick, ->] (0) -- (3);
	
	\draw [semithick, densely dotted] (0) -- (4);
	\draw [semithick, densely dotted] (0) -- (5);
	\draw [semithick, densely dotted] (0) -- (6);
	
\end{tikzpicture}
}
\tdplotsetmaincoords{0}{0}
\resizebox{!}{98 pt}{
\begin{tikzpicture}[tdplot_main_coords]
	\tdplotsetrotatedcoords{15}{-15}{-15}
	\begin{scope}[tdplot_rotated_coords]

	\coordinate (0) at (0,0,0);
	\coordinate (1) at (0,0.187592474085080,-0.982246946376846);
	\coordinate (2) at (-0.850650808352040,0.187592474085080,0.491123473188423);
	\coordinate (3) at (0.850650808352040,0.187592474085080,0.491123473188423);
	\coordinate (4) at (0, -0.794654472291766,-0.607061998206686);
	\coordinate (5) at (-0.525731112119134,-0.794654472291766,0.303530999103343);
	\coordinate (6) at (0.525731112119134,-0.794654472291766,0.303530999103343);
	\coordinate (7) at (0,-0.187592474085080,0.982246946376846);
	\coordinate (8) at (0.850650808352040,-0.187592474085080,-0.491123473188423);
	\coordinate (9) at (-0.850650808352040,-0.187592474085080,-0.491123473188423);
	\coordinate (10) at (0, 0.794654472291766,0.607061998206686);
	\coordinate (11) at (0.525731112119134,0.794654472291766,-0.303530999103343);
	\coordinate (12) at (-0.525731112119134,0.794654472291766,-0.303530999103343);
	
	\draw [densely dashed, very thin, lightgray] (1) -- (4);
	\draw [densely dashed, very thin, lightgray] (1) -- (8);
	\draw [densely dashed, very thin, lightgray] (1) -- (9);
	\draw [densely dashed, very thin, lightgray] (1) -- (11);
	\draw [densely dashed, very thin, lightgray] (1) -- (12);
	\draw [densely dashed, very thin, lightgray] (2) -- (9);
	\draw [densely dashed, very thin, lightgray] (4) -- (5);
	\draw [densely dashed, very thin, lightgray] (4) -- (6);
	\draw [densely dashed, very thin, lightgray] (4) -- (8);
	\draw [densely dashed, very thin, lightgray] (4) -- (9);
	\draw [densely dashed, very thin, lightgray] (5) -- (9);
	\draw [densely dashed, very thin, lightgray] (9) -- (12);
			
	\draw [thick, blue, ->] (0) -- (1);
	\draw [thick, red, ->] (0) -- (4);
	\draw [semithick, blue, densely dotted] (0) -- (8);
	\draw [semithick, blue, densely dotted] (0) -- (9);
	\draw [semithick, red, densely dotted] (0) -- (11);
	\draw [semithick, red, densely dotted] (0) -- (12);
	\draw [thick, red, ->] (0) -- (5);
	\draw [thick, red, ->] (0) -- (6);
	\draw [thick, blue, ->] (0) -- (2);
	\draw [thick, blue, ->] (0) -- (3);
	\draw [semithick, red, densely dotted] (0) -- (10);
	\draw [semithick, blue, densely dotted] (0) -- (7);
	
	\draw [darkgray, line width=0.3pt] (2) -- (5);
	\draw [darkgray, line width=0.3pt] (2) -- (7);
	\draw [darkgray, line width=0.3pt] (2) -- (10);
	\draw [darkgray, line width=0.3pt] (2) -- (12);
	\draw [darkgray, line width=0.3pt] (3) -- (6);
	\draw [darkgray, line width=0.3pt] (3) -- (7);
	\draw [darkgray, line width=0.3pt] (3) -- (8);
	\draw [darkgray, line width=0.3pt] (3) -- (10);
	\draw [darkgray, line width=0.3pt] (3) -- (11);
	\draw [darkgray, line width=0.3pt] (5) -- (6);
	\draw [darkgray, line width=0.3pt] (5) -- (7);
	\draw [darkgray, line width=0.3pt] (6) -- (7);
	\draw [darkgray, line width=0.3pt] (6) -- (8);
	\draw [darkgray, line width=0.3pt] (7) -- (10);
	\draw [darkgray, line width=0.3pt] (8) -- (11);
	\draw [darkgray, line width=0.3pt] (10) -- (11);
	\draw [darkgray, line width=0.3pt] (10) -- (12);
	\draw [darkgray, line width=0.3pt] (11) -- (12);

	\end{scope}
\end{tikzpicture}
}
\end{center}

\caption{
When $\tfrac{n}{d} \approx 2$, Theorem~\ref{thm:doubling} doubles an ETF of size $d\times n$ (left) to make another ETF of size $n \times 2n$ (right).
The double of a real ETF typically has complex entries, but 
it is real when $(d,n) = (2,3)$.
See Example~\ref{ex:icosahedron}.
}
\label{fig:icosahedron}
\end{figure}

\begin{example}
\label{ex:icosahedron}
Figure~\ref{fig:icosahedron} shows the $2 \times 3$ \textit{Mercedes--Benz} ETF together with its double of size $3 \times 6$.
Columns of the latter are vertices of a regular icosahedron.
Atypically, the doubled ETF has real entries in this case since $\beta = 1$ is real.
To produce the figure, we applied Theorem~\ref{thm:doubling}(b) using a row of all-ones for the Naimark complement.
The doubled ETF is made from two copies of Mercedes--Benz that are each ``lifted'' out of the $x_1x_2$-plane.
Here $b>0$ and $z<0$, so one copy lifts in the positive $x_3$-direction (blue) and the other lifts in the negative $x_3$-direction (red).
\end{example}

\begin{example}
Let $q = 3 \bmod 4$ be a prime power.
Then there is a complex ETF of size $\tfrac{q-1}{2} \times q$ associated with the Paley difference set of nonzero quadratic residues in the finite field of order~$q$~\cite{DF:07,SH:03,XZG:05}.
By Theorem~\ref{thm:doubling}, there is a complex ETF of size $q \times 2q$.
Taking $q=11$ shows there is a complex ETF of size~$11 \times 22$, which was unknown before now.
\end{example}

\begin{proof}[Proof of Theorem~\ref{thm:doubling}]
For (a), we apply Proposition~\ref{prop:sig}.
It suffices to verify $\Sigma^2 = (2n-1) I$.
Indeed,
\[
\Sigma^2 = 
\left[ \begin{array}{cc}
S^2 + (S+\beta I)(S + \overline{\beta} I) & 0 \\ 0 & (S+\overline{\beta} I)(S+\beta I) + S^2
\end{array} \right].
\]
The blocks on the diagonal coincide since polynomials of $S$ commute.
Both equal
\[
2S^2 + (2 \operatorname{Re} \beta)S + I = 2S^2 - 2cS + I,
\]
which simplifies to $(2n-1)I$ since $S^2 = cS + (n-1) I$.

For (b), we first verify $b >0$, that is, $\mu > \lambda$.
This is easy to check when $n = 2$ and $d = 1$.
Now assume $n \geq 3$.
Then $n \geq \sqrt{4n-3}$.
Subtracting~$1$ and dividing by $2(n-1)$ yields ${\tfrac{1}{2} \geq \frac{\sqrt{4n-3}-1}{2(n-1)}}$.
Applying~\eqref{eq:window around 2}, we find ${\tfrac{n}{d} -2 \geq -\tfrac{1}{2}}$, or ${\tfrac{n}{d} - 1 \geq \tfrac{1}{2}}$.
Dividing both sides by~$n-1$ shows
\[
\frac{n-d}{d(n-1)} \geq \frac{1}{2n-2} > \frac{1}{2n-1}.
\]
Thus, $\mu > \lambda$.

Next, fix an ETF
\[
\Phi = \left[ \begin{array}{cc} A & B \\ C & D \end{array} \right] \in \mathbb{C}^{n \times 2n}
\]
with signature matrix $\Sigma$, where $A,B \in \mathbb{C}^{d \times n}$ and $C,D \in \mathbb{C}^{(n-d) \times n}$.
Then $\Phi^* \Phi = I + \lambda \Sigma$, so that
\begin{equation}
\label{eq:blocks}
\left[ \begin{array}{cc}
A^* A + C^* C & A^* B + C^* D \\
B^* A + D^* C & B^* B + D^* D
\end{array} \right]
= 
\left[ \begin{array}{cc}
I + \lambda S & \lambda S+\lambda\beta I \\
\lambda S + \lambda\overline{\beta}I & I-\lambda S
\end{array} \right].
\end{equation}

To identify $A$ and $C$, first apply the relations $F^*F = I+\mu S$ and $G^*G = I-\nu S$ to observe that
\[
\left[ \begin{array}{cc} aF^* & bG^* \end{array} \right]
\left[ \begin{array}{c} aF \\ b G \end{array} \right]
= (a^2 + b^2) I + (a^2 \mu - b^2 \nu) S
= 
I + \lambda S
=
\left[ \begin{array}{cc} A^* & C^* \end{array} \right]
\left[ \begin{array}{c} A \\ C \end{array} \right].
\]
Consequently, there is a unitary $U \in \mathbb{C}^{2n \times 2n}$ for which $U \left[ \begin{array}{c} A \\ C \end{array} \right] = \left[ \begin{array}{c} aF \\ bG \end{array} \right]$.
Now $U \Phi$ is also an ETF with signature matrix $\Sigma$, and by replacing $\Phi$ with $U\Phi$ we may assume $A = aF$ and $C = bG$.

To identify $B$ and $D$, compare the top-right entries in~\eqref{eq:blocks} to find
\begin{equation}
\label{eq:top right}
aF^* B + bG^* D = \lambda S + \lambda \beta I.
\end{equation}
Multiplying on the left by $F$ gives
\[
a F F^* B + bFG^* D = \lambda FS + \lambda \beta F.
\]
Here, $FF^* = \tfrac{n}{d} I$, while $FG^* = 0$ since $F$ and $G$ are Naimark complements.
Meanwhile, $FS = \tfrac{1}{\mu}F(F^*F - I) = \tfrac{1}{\mu}( \tfrac{n}{d} - 1) F$.
Overall,
\[
\frac{an}{d} B + 0 = \frac{\lambda}{\mu} \Big( \frac{n}{d} - 1 \Big)F + \lambda \beta F,
\]
which rearranges to show $B = w F$.
A similar argument multiplying~\eqref{eq:top right} on the left by $G$ gives $D = zG$.
\end{proof}

\begin{remark}
The double of an ETF can be doubled again.
Doing this in practice with Theorem~\ref{thm:doubling}(b) requires a Naimark complement of $\Phi$.
Such a complement can be obtained by doubling $G$ (as opposed to $F$) while using the other choice of $\epsilon = \pm 1$.
Indeed, this gives an opposite signature matrix in~\eqref{eq:Sigma}.
\end{remark}

Repeated ETF doubling proves the following.

\begin{corollary}
If there is a real or complex ETF of size $d \times 2d$, then there is a complex ETF of size $2^k d \times 2(2^{k} d)$ for every~$k$.
\end{corollary}

In particular, to prove Conjecture~\ref{conj:dx2d} it would be enough to produce a $d \times 2d$ complex ETF for every odd~$d$.
This is similar to the skew Hadamard conjecture, where it would be enough to produce a skew Hadamard of order $4k$ for every odd~$k$.

\begin{remark}
Similar doubling theorems apply for Welch-type optimal packings of higher-dimensional subspaces.
Specifically, assume $d$, $r$, and $n$ satisfy
\[
c= (nr - 2d) \sqrt{ \frac{ n-1 }{d(nr-d) } } \in [-1,1],
\]
and suppose there is an \textit{equi-isoclinic tight fusion frame}~\cite{FJMW:17} consisting of $n$ subspaces with dimension $r$ in $\mathbb{R}^d$ or $\mathbb{C}^d$, briefly, an $\operatorname{EITFF}(d,r,n)$.
Then a procedure just like that of Theorem~\ref{thm:doubling}(a) produces a complex $\operatorname{EITFF}(rn,r,2n)$.
The same holds for \textit{equi-chordal tight fusion frames}~\cite{FJMW:17}.
\end{remark}

\section{ETFs from skew Hadamards}
To prove the main result, we double ETFs from a particular class.
Let $H$ be a skew Hadamard.
If $D$ is any diagonal matrix with $\pm 1$ diagonal entries, then $DHD$ is also a skew Hadamard.
The \textbf{normalization} of~$H$ is the unique such $DHD$ with all $+1$ in the top row.

\begin{definition}
\label{defn:coreAdj}
Given a skew Hadamard matrix $H$ of order $m>1$, decompose its normalization as a block array
\[
\left[ \begin{array}{rc}
1 & \mathbf{1}^\top \\
-\mathbf{1} & A - A^\top+I
\end{array}
\right],
\]
where $\mathbf{1}$ is a column vector of ones and $A$ is a $(0,1)$-matrix of order $m-1$ with zero diagonal.
We call $A$ the \textbf{core adjacency matrix} of $H$.
\end{definition}

This terminology is our own.
In words, the core adjacency matrix gives the off-diagonal locations of $+1$ in the ``core'' of the normalization.

\begin{proposition}[\cite{S:08}]
\label{prop:core}
Suppose there is a skew Hadamard of order~$m>2$,
and let $A$ be its core adjacency matrix.
For
\[
\alpha = -\frac{1}{\sqrt{m}} + i \sqrt{1 - \frac{1}{m}},
\]
$S=\alpha A + \overline{\alpha} A^\top$ is the signature matrix of an $\tfrac{n-1}{2} \times n$ complex ETF with $n = m-1$.\end{proposition}

In effect, $S$ is a modified ``core'' of the normalized Hadamard, where off-diagonal $+1$ entries are replaced with $\alpha$ and $-1$ entries are replaced with $\overline{\alpha}$.
We attribute Proposition~\ref{prop:core} to Strohmer~\cite{S:08}, who proved a conjecture of Renes~\cite{R:07}.
Indeed, our description of the signature matrix can be gleaned from the proof of Theorem~1.1 in~\cite{S:08}.
Here we give a new proof based on the theory of association schemes~\cite{BI:84}.

\begin{proof}[Proof of Proposition~\ref{prop:core}]
It is well known that $I$, $A$, and $A^\top$ are the adjacency matrices of a commutative association scheme, where each of $A$ and $A^\top$ has exactly $\tfrac{n-1}{2}$ ones per row, and where 
\[
A^2 = \tfrac{n-3}{4} A + \tfrac{n-1}{4} A^\top,
\qquad
(A^\top)^2 = \tfrac{n-3}{4} A^\top + \tfrac{n-1}{4} A,
\]
\[
A A^\top = A^\top A = \tfrac{n-3}{2} I + \tfrac{n-3}{4} A + \tfrac{n-3}{4} A^\top.
\]
(See~\cite{H:00}, for instance.)
Now one can either do a messy calculation to verify that 
\[
S^2 = \tfrac{2}{\sqrt{n+1}} S + (n-1) I,
\]
or one can ignore the relations and argue as follows.

The adjacency matrices form a basis for the adjacency algebra $\operatorname{span}\{I,A,A^\top\}$, and when any matrix from the algebra is expanded in this basis, the coefficients are just the matrix entries.
Consider a spectral basis for the adjacency algebra consisting of mutually orthogonal primitive idempotents $\tfrac{1}{n}J$, $E_1$, and $E_2$ (projections onto the eigenspaces of $A$). 
Since $E_1 = E_1^*$, it must expand as $E_1 = zA+\overline{z}A^\top+\tfrac{d}{n} I$ for some $z\in \mathbb{C}$, where $d = \operatorname{tr} E_1 = \operatorname{rank} E_1 < n-1$.
Then $\tfrac{n}{d} E_1$ is the Gram matrix of an ETF with size $d \times n$, and considering the phases of its off-diagonal entries, the signature matrix is $\tfrac{z}{|z|} A + \tfrac{\overline{z}}{|z|} A^\top$.

It remains to show $d = \tfrac{n-1}{2}$ and $\tfrac{z}{|z|} = \alpha$.
For the dimension, observe that $z \not\in \mathbb{R}$ since $E_1 \not\in \operatorname{span}\{I,J\}$ by virtue of its rank.
The entrywise complex conjugate $\overline{E_1} \neq E_1$ must be another primitive idempotent, and so $E_2 = \overline{E_1}$ also has rank~$d$.
Since the ranks of the primitive idempotents add up to $n$, we conclude that $d = \tfrac{n-1}{2}$.

Next, the Frobenius norm of the orthogonal projection $E_1$ is the square root of its rank, and so
\[ (n^2-n)|z|^2 + n( \tfrac{d}{n} )^2 = d, \]
that is,
\[
|z|^2 = \frac{d(n-d)}{n^2(n-1)} = \frac{n+1}{4n^2}.
\]
Meanwhile, the entries in any row of $E_1$ add up to $0$ since $E_1J = 0$, and so
\[
z \frac{n-1}{2} + \overline{z} \frac{n-1}{2} + \frac{d}{n} = 0,
\]
that is,
\[
\operatorname{Re} z = - \frac{d}{n(n-1)} = -\frac{1}{2n}.
\]
It follows that $\operatorname{Im} z = \pm \frac{\sqrt{n}}{2n}$ and
\[
\frac{z}{|z|} = - \frac{1}{\sqrt{n+1}} \pm i \sqrt{ \frac{n}{n+1} }.
\]
After replacing $E_1$ with $\overline{E_1}$ if necessary, we find $\frac{z}{|z|} = \alpha$.
\end{proof}

When taken together, Theorem~\ref{thm:doubling}(a) and Proposition~\ref{prop:core} imply the following constructive version of the main result.

\begin{corollary}
\label{cor:main const}
Suppose there is a skew Hadamard of order~$m>2$, and let $A$ be its core adjacency matrix (Definition~\ref{defn:coreAdj}).
For
\[
\alpha = -\frac{1}{\sqrt{m}} + i \sqrt{1 - \frac{1}{m}}
\qquad
\text{and}
\qquad
\beta = -\frac{2}{\sqrt{m}} + i \sqrt{ 1 - \frac{4}{m} },
\]
then
\[
\Sigma=
\left[ \begin{array}{cc}
\alpha A + \overline{\alpha} A^\top & \alpha A + \overline{\alpha} A^\top + \beta I \\
\alpha A + \overline{\alpha} A^\top + \overline{\beta} I & - \alpha A - \overline{\alpha} A^\top
\end{array} \right]
\]
is the signature matrix of a complex ETF with size $(m-1) \times 2(m-1)$.
\end{corollary}

\begin{remark}
In Corollary~\ref{cor:main const}, $\Sigma$ is a self-adjoint complex conference matrix of order~$2(m-1)$, and $\Sigma+i I$ is a complex Hadamard matrix of constant diagonal~\cite{S:13}.
\end{remark}

The following accounts for both Hadamard and ETF doubling.

\begin{corollary}
\label{cor:double main}
If there is a skew Hadamard of order $m$, then for every $j$ and $k$ there is an ETF of size $d \times 2d$ with $d = 2^k (2^j m-1)$.
\end{corollary}

\begin{example}
For each $d$ below, no skew Hadamard of order~$2d$ is presently known~\cite{D:23}:
\[
d = 214,\, 238,\, 334,\, 358,\, 382.
\]
In each case, we can nevertheless produce an ETF of size 
$d \times 2d$,
as in Proposition~\ref{prop:m/2}.
Indeed, a skew Hadamard of order 
$\tfrac{d}{2}+1$
is known.
By the main result, there is an ETF of size 
$\tfrac{d}{2} \times d$.
Then doubling gives one of size
$d \times 2d$.
For each $d$ above, a $d \times 2d$ ETF is included as an ancillary file with the arXiv version of this paper.
\end{example}

\section*{Acknowledgments}

The authors thank Matt Fickus and Dustin Mixon for insightful comments.

\bigskip

\bibliographystyle{abbrv}
\bibliography{refs}

\begin{thebibliography}{10}

\bibitem{BI:84}
E.~Bannai and T.~Ito.
\newblock {\em Algebraic combinatorics {I}}.
\newblock Benjamin/Cummings, 1984.

\bibitem{CHS:96}
J.~H. Conway, R.~H. Hardin, and N.~J.~A. Sloane.
\newblock Packing lines, planes, etc.: {P}ackings in {G}rassmannian spaces.
\newblock {\em Experiment. Math.}, 5(2):139--159, 1996.

\bibitem{CK:07}
R.~Craigen and H.~Kharaghani.
\newblock {H}adamard matrices and {H}adamard designs.
\newblock In C.~J. Colbourn and J.~H. Dinitz, editors, {\em Handbook of
  Combinatorial Designs}, pages 273--280. CRC, Boca Raton, FL, 2nd edition,
  2007.

\bibitem{DGS:75}
P.~Delsarte, J.~M. Goethals, and J.~J. Seidel.
\newblock Bounds for systems of lines, and {J}acobi polynomials.
\newblock {\em Philips Res. Repts.}, 30:91*--105*, 1975.
\newblock Issue in honor of C. J. Bouwkamp.

\bibitem{DF:07}
C.~Ding and T.~Feng.
\newblock A generic construction of complex codebooks meeting the {W}elch
  bound.
\newblock {\em IEEE Trans. Inform. Theory}, 53(11):4245--4250, 2007.

\bibitem{D:23}
D.~Z. Dokovi\'{c}.
\newblock Skew-{H}adamard matrices of order 276.
\newblock arXiv:2301.02751.

\bibitem{FJMW:17}
M.~Fickus, J.~Jasper, D.~G. Mixon, and C.~E. Watson.
\newblock {A brief introduction to equi-chordal and equi-isoclinic tight fusion
  frames}.
\newblock In {\em Wavelets and Sparsity XVII}, volume 10394, page 103940T.
  SPIE, 2017.

\bibitem{FM:15}
M.~Fickus and D.~G. Mixon.
\newblock Tables of the existence of equiangular tight frames.
\newblock arXiv:1504.00253v2.

\bibitem{G:21}
O.~Gritsenko.
\newblock On strongly regular graph with parameters (65; 32; 15; 16).
\newblock arXiv:2102.05432.

\bibitem{H:00}
A.~Hanaki.
\newblock Skew-symmetric {H}adamard matrices and association schemes.
\newblock {\em SUT J. Math.}, 36(2):251--258, 2000.

\bibitem{HP:04}
R.~B. Holmes and V.~I. Paulsen.
\newblock Optimal frames for erasures.
\newblock {\em Linear Algebra Appl.}, 377:31--51, 2004.

\bibitem{JKM:19}
J.~Jasper, E.~J. King, and D.~G. Mixon.
\newblock {Game of Sloanes: {B}est known packings in complex projective space}.
\newblock In {\em Wavelets and Sparsity XVIII}, volume 11138, page 111381E.
  SPIE, 2019.

\bibitem{KS:08}
C.~Koukouvinos and S.~Stylianou.
\newblock On skew-{H}adamard matrices.
\newblock {\em Discrete Math.}, 308(13):2723--2731, 2008.

\bibitem{LS:73}
P.~W.~H. Lemmens and J.~J. Seidel.
\newblock Equiangular lines.
\newblock {\em J. Algebra}, 24:494--512, 1973.

\bibitem{L:82}
V.~I. Levenstein.
\newblock Bounds on the maximal cardinality of a code with bounded modules of
  the inner product.
\newblock {\em Soviet Math. Dokl.}, 25:526--531, 1982.

\bibitem{LHS:03}
D.~J. Love, R.~W. Heath, Jr., and T.~Strohmer.
\newblock Grassmannian beamforming for multiple-input multiple-output wireless
  systems.
\newblock {\em IEEE Trans. Inform. Theory}, 49:2735--2747, 2003.

\bibitem{R:55}
R.~A. Rankin.
\newblock The closest packing of spherical caps in {$n$} dimensions.
\newblock {\em Proc. Glasgow Math. Assoc.}, 2:139--144, 1955.

\bibitem{R:07}
J.~M. Renes.
\newblock Equiangular tight frames from {P}aley tournaments.
\newblock {\em Linear Algebra Appl.}, 426(2-3):497--501, 2007.

\bibitem{S:08}
T.~Strohmer.
\newblock A note on equiangular tight frames.
\newblock {\em Linear Algebra Appl.}, 429(1):326--330, 2008.

\bibitem{SH:03}
T.~Strohmer and R.~W. Heath, Jr.
\newblock Grassmannian frames with applications to coding and communication.
\newblock {\em Appl. Comput. Harmon. Anal.}, 14(3):257--275, 2003.

\bibitem{S:13}
F.~Sz\"{o}ll\H{o}si.
\newblock Complex {H}adamard matrices and equiangular tight frames.
\newblock {\em Linear Algebra Appl.}, 438(4):1962--1967, 2013.

\bibitem{T:08}
J.~A. Tropp.
\newblock On the conditioning of random subdictionaries.
\newblock {\em Appl. Comput. Harmon. Anal.}, 25(1):1--24, 2008.

\bibitem{TDHS:05}
J.~A. Tropp, I.~S. Dhillon, R.~W. Heath, Jr., and T.~Strohmer.
\newblock Designing structured tight frames via an alternating projection
  method.
\newblock {\em IEEE Trans. Inform. Theory}, 51(1):188--209, 2005.

\bibitem{W:71a}
J.~Wallis.
\newblock A skew-{H}adamard matrix of order {$92$}.
\newblock {\em Bull. Austral. Math. Soc.}, 5:203--204, 1971.

\bibitem{W:71b}
J.~Wallis.
\newblock Some {$(1,\,-1)$} matrices.
\newblock {\em J. Combinatorial Theory Ser. B}, 10:1--11, 1971.

\bibitem{W:74}
L.~Welch.
\newblock Lower bounds on the maximum cross correlation of signals.
\newblock {\em {IEEE} Trans. Inform. Theory}, 20(3):397--399, 1974.

\bibitem{XZG:05}
P.~Xia, S.~Zhou, and G.~B. Giannakis.
\newblock Achieving the {W}elch bound with difference sets.
\newblock {\em IEEE Trans. Inform. Theory}, 51(5):1900--1907, 2005.

\end{thebibliography}

\end{document}